\documentclass[11pt]{article}  
\usepackage[utf8]{inputenc}
\usepackage[pagebackref=true]{hyperref}
\usepackage{amssymb,upgreek,amsmath,bm,bbm,amsthm, graphicx, epstopdf, tikz,xcolor,tocbibind}
\usepackage{algpseudocode,algorithm,algorithmicx}
\usepackage{enumitem}
\usepackage[normalem]{ulem}
\usepackage[justification=centering]{subfig}
\usepackage[export]{adjustbox}

\usepackage{mystyle}

\graphicspath{{figures/}}
\usepackage[OT2,T1]{fontenc}
\DeclareSymbolFont{cyrletters}{OT2}{wncyr}{m}{n}
\DeclareMathSymbol{\Sha}{\mathalpha}{cyrletters}{"58}
\usepackage{geometry}
 
\usepackage[page]{appendix}
\definecolor{DarkBlue}{rgb}{0.15,0.15,0.55}
\hypersetup{linktocpage=true,colorlinks=true,allcolors=DarkBlue}
\usepackage[capitalise]{cleveref}
\usepackage{titlesec}
\titleformat*{\paragraph}{\itshape\mdseries}
\newcommand{\Ty}{\bm{\mathrm{T}}_{\bm{y}}}
\newcommand{\Vx}{\bm{\mathrm{V}}_{\bm{x}}}
\newcommand{\Da}{\bm{\mathrm{D}}_{\bm{a}}}
\newtheorem*{theorem*}{Theorem}

\title{On the Uniqueness of Solutions for \\ the Basis Pursuit in the Continuum}

\author{Thomas Debarre\footnote{EPFL, BIG, Switzerland, thomas.debarre@epfl.ch}, Quentin Denoyelle\footnote{Université Paris Cité, CNRS, MAP5, F-75006 Paris, France, quentin.denoyelle@u-paris.fr}, Julien Fageot\footnote{EPFL, LCAV, Switzerland, julien.fageot@epfl.ch}}
\begin{document} 

\maketitle


\begin{abstract}
This paper studies the continuous-domain inverse problem of recovering Radon measures on the one-dimensional torus from low-frequency Fourier coefficients, where $K_c$ is the cutoff frequency. Our approach consists in minimizing the total-variation norm among all Radon measures that are consistent with the observations. We call this problem the basis pursuit in the continuum (BPC). We characterize the solution set of (BPC) in terms of uniqueness and describe its sparse solutions which are sums of few signed Dirac masses. The characterization is determined by the spectrum of a Toeplitz and Hermitian-symmetric matrix that solely depends on the observations. More precisely, we prove that (BPC) has a unique solution if and only if this matrix is neither positive definite nor negative definite. If it has both a positive and negative eigenvalue, then the unique solution is the sum of at most $2K_c$ Dirac masses, with at least one positive and one negative weight. If this matrix is positive (respectively negative) semi-definite and rank deficient, then the unique solution is composed of a number of Dirac masses equal to the rank of the matrix, all of which have nonnegative (respectively nonpositive) weights. Finally, in cases where (BPC) has multiple solutions, we demonstrate that there are infinitely many solutions composed of $K_c+1$ Dirac masses, with nonnegative (respectively nonpositive) weights if the matrix is positive (respectively negative) definite.
\end{abstract}


%
\section{Introduction}

In recent years, total-variation\footnote{ In this paper, the ``total variation'' is understood in the sense of measure theory and should not be confused with the concept of BV functions \cite{rudin1992nonlinear}.} (TV) regularization techniques have proved to be very fruitful to solve continuous-domain linear inverse problems with a sparsity prior.
They provide a general framework for the the recovery of sparse continuous-domain signals (\emph{e.g.}, Dirac streams or splines) from possibly corrupted finite-dimensional measurements. 
Such techniques rely on solid theoretical foundations~\cite{de2012exact,Candes2013Super,Bredies2013inverse,Duval2015exact,Unser2017splines}, but also on many algorithmic advances~\cite{Bredies2013inverse,boyd2017alternating,flinth2019linear}, and have found various data-science applications~\cite{boyd2017alternating,courbot2019sparse,courbot2020fast,simeoni2021functional}. It is well known that their discrete-domain counterparts (\emph{i.e.}, $\ell^1$ regularization methods) lead to variational problems whose solutions are not necessarily unique~\cite{Tibshirani2013lasso}. Our goal in this paper is to provide a systematic study of the uniqueness and of the main properties of the solutions of the total-variation norm minimization problem when the low-frequency Fourier coefficients of the signal are prescribed.

\subsection{Reconstruction via Total-Variation Minimization}

Let $\TT = \R / 2\pi \Z$ be the $1$D torus. We study the problem of recovering \emph{real} Radon measures $w\in\Radon$ defined on the domain $\TT$ based on their low-frequency Fourier coefficients 
\begin{align}
	\forall k\in\ZZ, \ |k|\leq K_c, \quad \widehat w [k]= \frac{1}{2\pi} \int_{0}^{2\pi} \mathrm{e}^{- \mathrm{i} k t} \mathrm{d} w (t) =  y_k,
\end{align}
where $K_c\in\NN$ denotes the cutoff frequency and $\V y = (y_{-K_c},\ldots,y_{-1},y_0,y_1,\ldots,y_{K_c})\in\CC^{2K_c+1}$ are the observations. As the measures $w\in\Radon$ are real, we must have $y_0\in\RR$ and for all $k\in\{1,\ldots,K_c\}$, $y_k = \overline{y_{-k}}$.
Therefore, the Toeplitz matrix 
\begin{align}\label{eq:defTy}
	\Ty = 2\pi
	\begin{pmatrix}
		y_0 & y_1 & \cdots & \cdots & y_{K_c} \\
		y_{-1} & y_0 & y_1 & \cdots & y_{K_c-1} \\
		\vdots & \ddots & \ddots & \ddots & \vdots \\
		y_{-K_c+1} & \cdots & y_{-1} & y_{0} & y_1 \\
		y_{-K_c} & \cdots & \cdots & y_{-1} & y_{0} \\
	\end{pmatrix}\in\CC^{(K_c+1)\times (K_c+1)},
\end{align}
which is at the core of our main contribution, is also Hermitian symmetric. Moreover, the observation vector $\bm{y}\in\CC^{2K_c+1}$ has only $2K_c + 1$ (real) degrees of freedom.

The recovery of a periodic Radon measure from finitely many observations is clearly an ill-posed problem. Therefore, we choose to formulate the reconstruction task as a regularized optimization problem using a sparsity prior. More precisely, in this paper, we consider the problem
\begin{align}\label{intro:BPC}
	\underset{w\in\Radon, \ \nuf(w)=\V y}{\inf} \ \mnorm{w},
\end{align}
where $\nuf:\Radon\to\CC^{2K_c+1}$ is the measurement vector of Fourier coefficients
\begin{align} \label{eq:nuf}
    \forall w\in\Radon, \quad \nuf(w) = (\widehat{w}[-K_c],\ldots, \widehat{w}[0], \ldots, \widehat{w}[K_c]),
\end{align}
and $\mnorm{\cdot}$ is the total-variation (TV) norm on the space of Radon measures $\Radon$. The TV norm can be seen as an extension of the $\ell^1$ norm to the continuous domain. The choice of the TV norm promotes sparse continuous-domain reconstruction, and has recently received a lot of attention (see Section~\ref{sec:relatedworks}). In the following, we name the Problem~\eqref{intro:BPC}, the \emph{basis pursuit in the continuum} (BPC), in reference to its discrete-domain counterpart, the basis pursuit in~\cite{chen2001atomic}.

\subsection{Related Works}\label{sec:relatedworks}

\paragraph{Optimization over Radon measures.}
The historical motivation to consider the total-variation norm as a regularization was to extend discrete $\ell^1$ regularization techniques, used in the theory of compressed sensing to recover sparse vectors~\cite{Donoho2006,Elad10,Foucart2013mathematical,Unser2016representer}, for the recovery of continuous-domain Dirac masses. The goal is to recover point sources, modelled as a sum of Dirac masses, from finitely many measurements. This question has received considerable attention in the 21st century, including methods that are not based on TV regularization, such as finite rate of innovation (FRI) techniques~\cite{Vetterli2002FRI} and Prony's methods~\cite{schmidt1986multiple,plonka2014prony}. Several data-science problems can be formulated as a continuous-domain Dirac recovery problem, including radio-astronomy~\cite{pan2016towards}, super-resolution microscopy~\cite{boyd2017alternating}, or 3D image deconvolution~\cite{courbot2020fast}. 

The study of optimization problems over the space of Radon measures can be traced back to the pioneering works of Beurling~\cite{beurling1938integrales}, where Fourier-domain measurements were also considered. In the early 2010's, the works of De Castro and Gamboa~\cite{de2012exact}, Candès and Fernandez-Granda~\cite{Candes2013Super,candes2014towards}, and Bredies and Pikkarainen~\cite{Bredies2013inverse} considered optimization tasks of the form \eqref{intro:BPC} (or its penalized version), with both theoretical analyses and novel algorithmic approaches to recover a sparse-measure solution, in the continuum~\cite{Bredies2013inverse,eftekhari2013greed,boyd2017alternating,elvira2021omp,flinth2019linear,traonmilin2020projected,courbot2020fast,chizat2021sparse}. The existence of sparse-measure solutions, \emph{i.e.}, solutions of the form $\sum_{k=1}^K a_k \dirac{x_k}$, where $K\in\NN^*$, $a_k\in\RR$, and $\dirac{x_k}$ is the Dirac mass at the location $x_k$, seems to have been proven for the first time in~\cite{beurling1938integrales} and was later improved by Fisher and Jerome in~\cite{fisher1975spline}. Since then, a remarkable revival around TV optimization has occurred recently~\cite{Azais2015Spike,duval2017sparseI,duval2017sparseII,flinth2019linear,poon2019multidimensional}.

\paragraph{Algorithms for TV optimization.}
The numerical resolution of optimization problems based on the total-variation norm has been largely studied in the literature. If we do not contribute to these computational aspects in this work, we briefly recall how they have been treated in practice.  We can divide the different strategies to solve the basis pursuit in the continuum~\eqref{intro:BPC} numerically, or its penalized version (often called the BLASSO~\cite{Bredies2013inverse}), into three main approaches. A first one consists in discretizing the spatial domain, which converts the infinite-dimensional convex problem into a finite dimensional one, known as the basis pursuit~\cite{chen2001atomic} or the LASSO~\cite{tibshirani1996regression}. Many standard solvers exist to solve these problems, such as block-coordinate descent algorithms~\cite{tseng2001convergence}, homotopy algorithms~\cite{osborne2000new}, or proximal algorithms~\cite{combettes2005fb, beck2009fista}. A second one is based on reformulating the problem as a semidefinite program. This has been done for the basis pursuit in the continuum when the forward operator is a convolution with the ideal low-pass filter (which is the setup of this paper) in~\cite{candes2014towards}. One can expect to recover the positions of the Dirac masses exactly since no discretization is performed. However, these techniques are limited to the case where the dual problem involves trigonometrical polynomials, and usually to the one-dimensional setting (although some extensions exist in higher dimensions~\cite{de2016exact,catala2019low}). A last line of approaches tackles the BLASSO problem directly over the space of Radon measures. They involve for instance gradient descents and gradient flows~\cite{chizat2021sparse}, or the Frank-Wolfe algorithm~\cite{Frank1956algorithm} which leads to greedy methods~\cite{Bredies2013inverse,boyd2017alternating} that can achieve finite-time convergence in some cases~\cite{denoyelle2019sliding}.

\paragraph{Uniqueness results for TV optimization.}

It is well know that finite-dimensional $\ell^1$-regularization, of which the total-variation norm for Radon measures is the continuous-domain generalization, can lead to nonunique solutions~\cite{Tibshirani2013lasso,ali2019generalized}. This is also the case for TV regularization. However, under some assumptions uniqueness may hold. Many uniqueness results for constrained or penalized TV-based optimization problems have thus been given in the literature, but from different perspectives than the one studied in this paper. 

Usually, the underlying assumption is that the observations $\V y$ are generated via a sparse measure $w_0$ and the question becomes whether it is possible to uniquely recover $w_0$, either exactly or in a stable way. In~\cite{de2012exact}, de Castro and Gamboa introduced the concept of \emph{extrema Jordan type measure} (see \cite[Definition 1]{de2012exact}), which gives sufficient conditions on a given signed (with positive and negative weights) sparse measure to be the unique solution of a TV-based optimization problem. They also proved that when the input measure is nonnegative, $k$-sparse and the number of measurements is greater or equal than $2k+1$, then it is the unique solution of~\eqref{intro:BPC} if the measurement operator is defined from a T-system. Note that it has recently been proved~\cite{pouchol2020ml,pouchol2020linear,eftekhari2021sparse,eftekhari2021stable} that in the nonnegativity setting, a $k$-sparse nonnegative measure can be uniquely recovered from at least $2k+1$ measurements and a nonnegativity constraint without the need for TV minimization.

Candès and Fernandez-Granda also studied the super-resolution problem of recovering a ground-truth sparse Radon measure $w_0$ from its low-frequency measurements. They have shown that if the minimal distance between the spikes of $w_0$ is large enough, then~\eqref{intro:BPC} has a unique solution, which is $w_0$ itself~\cite[Theorem 1.2]{Candes2013Super}. Duval and Peyré identified the so-called \emph{nondegenerate source condition}~\cite[Definition 5]{Duval2015exact}, under which the uniqueness of the reconstruction together with the recovery of the support of the underlying ground-truth sparse measure are shown. These results are based on the key notion of dual certificates, which also play an important role in our work. This notion has been introduced for discrete compressed sensing problems in~\cite{candes2010probabilistic} and connected to TV-based optimization problems in~\cite{de2012exact}.
        
All these works are clearly related to this paper. However, the approach we propose here is different: we aim at characterizing the cases of uniqueness directly over the measurement vector $\bm{y}$, and we are agnostic to the ground-truth signal that generated it. 
The closest work in this direction is our recent publication~\cite{debarre2020sparsest}. We provide a full description of the solution set of non-periodic TV optimization problems with a regularization operator $\mathrm{D}^2$, where $\Op D$ is the derivative operator, and spatial sampling measurements.

\paragraph{The problem of moments.}This problem, or its extension the \emph{generalized problem of moments}~\cite{lasserre2008semidefinite,lasserre2009moments}, is a classical one where one seeks to recover a measure from a sequence of its (potentially generalized) moments. This problem covers many applications in various fields, including Fourier-domain measurements, which are simply  moments of trigonometrical polynomials. As a result, the tools developed in this domain can be harnessed, for example, to solve numerically~\eqref{intro:BPC} using semi-definite programming formulations by generalizing~\cite{Candes2013Super} to the multivariate case~\cite{de2016exact,catala2019low}. Moreover, there are many existence results for the problem of moments. In~\cite[Theorem 6.12]{curto1991recursiveness}, Curto and Fialkow prove the existence of a sparse nonnegative measure $w_0\in\Radon$ solution of the truncated trigonometric moment problem (as referred to in~\cite{curto1991recursiveness})
\begin{align}
	w\in\Radon, \quad \V y = \nuf(w),
\end{align}
if and only $\Ty$ is positive semi-definite. Furthermore, solution $w_0$ has $\rank(\Ty)$ Dirac masses. Additionally, when $\Ty$ is invertible, there exist infinitely many such $w_0$. When it is not, $w_0$ is unique. The contributions of~\cite{curto1991recursiveness} and their relation to ours are be discussed more precisely in Remark~\ref{rem:curto}.

\subsection{Contributions}\label{sec:contrib}

The existence of a solution to problem~\eqref{intro:BPC} is well established and can be obtained from the direct method in the calculus of variations. Moreover, it is known that there is always at least one sparse solution composed of at most $2K_c+1$ Dirac masses~\cite{fisher1975spline}. However, the solution is in general not unique (the simplest case of nonuniqueness is with $K_c=0$), and previous works studying this question often start by assuming that the measurements come from a particular input signal. In this paper, we focus on
\begin{enumerate}
\item characterizing all the cases of uniqueness for~\eqref{intro:BPC};
\item describing the sparse solutions, \emph{e.g.} the bounds on the number of Dirac masses and the signs of the weights;
\end{enumerate}
from simple conditions depending \emph{only} on the measurement vector $\V y$, involving in particular the matrix $\Ty$. The novelty of our approach thus lies in the fact that we are able to provide relatively deep information on the form of the sparse solutions of~\eqref{intro:BPC} and answer in all cases the question of uniqueness agnostically to the input signal that generated $\V y$.

Our first contribution, Theorem~\ref{thm:equivalencegame} in Section~\ref{sec:criterionBPC}, provides several equivalent conditions for the solution set of~\eqref{intro:BPC} to be composed of only nonnegative measures. We prove that if one of these conditions is not satisfied \emph{and} $y_0\geq 0$, then there is a unique solution composed of at most $2K_c$ Dirac masses, with at least one positive and one negative weight. One of these equivalent conditions is used, in Corollary~\ref{cor:unique-y0-small}, to formulate a simple criterion on the magnitudes of the entries of $\V y$ which is a sufficient condition for uniqueness. Theorem~\ref{thm:equivalencegame} can of course be readily adapted to the nonpositive case.

Theorem~\ref{thm:equivalencegame} does not cover all the situations that may arise since it does not adjudicate the uniqueness of a solution to~\eqref{intro:BPC} when the solution set is composed of only nonnegative (or nonpositive) measures.
This limitation is tackled in our main contribution, Theorem~\ref{thm:main}, in Section~\ref{sec:uniqueness-result}, which we state below. We recall that the Toeplitz matrix $\Ty$ is given in~\eqref{eq:defTy}.

\begin{theorem*}
The solution set of problem~\eqref{intro:BPC} can be characterized as follows:
\begin{enumerate}[label=\alph*)]
\item If $\Ty$ has at least one negative and one positive eigenvalue, then~\eqref{intro:BPC} has a unique solution composed of at most $2K_c$ Dirac masses, with at least one positive and one negative weight;
\item If $\Ty$ is positive, respectively negative, semi-definite and $\rank(\Ty)< K_c + 1$, then~\eqref{intro:BPC} has a unique solution composed of exactly $\rank(\Ty)$ positive, respectively negative, Dirac masses;
\item If $\Ty$ is positive, respectively negative, definite (which implies that $\rank(\Ty)=K_c+1$), then~\eqref{intro:BPC} has infinitely many solutions, and among sparse solutions none with less than $K_c+1$ Dirac masses and uncountably many of them composed of $K_c+1$ positive, respectively negative, Dirac masses.
\end{enumerate}
\end{theorem*}

This theorem provides information on the form of the sparse solutions, from conditions on the spectrum of the Hermitian Toeplitz matrix $\Ty$. Moreover, all possible scenarios are covered since all vectors $\V y\in\RR\times\CC^{K_c}$ fall in one single case of the theorem. Therefore, it also yields a final answer to the question of uniqueness.

In the case when $\Ty$ is positive (or negative) semi-definite and rank deficient, our result can be derived from~\cite[Theorem 6.12]{curto1991recursiveness}, once we admit that any nonnegative (or nonpositive) measures that match the observations is a solution of~\eqref{intro:BPC}\footnote{For a more detailed discussion on this matter, see Remark~\ref{rem:curto}.}. Our contributions are therefore closely related to this literature. However our framework  is also different, and one of the novelty of our work lies in the bridges made between these related results in order to answer the question of uniqueness for~\eqref{eq:BPC}. Consequently, we also choose to present, as much as possible, our contributions in a self-contained fashion. The proofs are based upon known tools from the field of point-source recovery:
\begin{itemize}
\item dual certificates (see Proposition~\ref{prop:certification});
\item the Herglotz theorem (see Proposition~\ref{prop:herglotz}), which characterizes positive sparse measures in terms of their Fourier coefficients;
\item the Caratheodory-Pisarenko-Fejer decomposition~\cite{caratheodory1911zusammenhang,pisarenko1973retrieval,yang2018frequency} (see Proposition~\ref{prop:CFP}), which leads to algorithms such as MUSIC~\cite{schmidt1986multiple}.
\end{itemize}


%
\section{Mathematical Preliminaries}\label{sec:maths}
    
In Section~\ref{sec:optim-problem}, we introduce the mathematical background of this paper. We also present the optimization problem of interest. In Section~\ref{subsec:duality}, we then remind the main tools from duality theory that we use for studying the total-variation minimization problem.

\subsection{Periodic Radon Measures and basis pursuit in the continuum}\label{sec:optim-problem}

\paragraph{Radon measures.}Let $\Radon$ be the space of periodic \emph{real} Radon measures. By the Riesz-Markov theorem~\cite{gray1984shaping}, $\Radon$ is the continuous dual of the space $\Cc(\TT)$ of continuous periodic real functions endowed with the supremum norm $\normi{\cdot}$. The total-variation norm on $\Radon$ is the dual norm associated to $(\Cc(\TT),\normi{\cdot})$ and is thus given by
\begin{align} \label{eq:tvnorm}
	\mnorm{w} = \sup_{\varphi \in \Cc(\TT), \ \normi{\varphi} \leq 1} \dotp{w}{\varphi}.
\end{align}
The normed space $(\Radon,\mnorm{\cdot})$ is then a Banach space.
The duality product between $w \in \Radon$ and $\varphi \in \Cc(\TT)$ is given by
\begin{equation} \label{eq:dualityproduct}
    \langle w , \varphi \rangle = \frac{1}{2\pi} \int_{0}^{2\pi} \varphi(t) \mathrm{d} w(t).
\end{equation}

The total-variation norm can be seen as the extension to the continuum of the $\ell^1$ norm for vectors, as for any sparse measure
\begin{align}
	w_{\V a,\V x} = \sum_{k=1}^K a_k\dirac{x_k}, \qwithq K\in\N^*, \forall k\in\{1,\ldots,K\}, \ a_k\in\RR, \ x_k\in\TT,
\end{align}
where the locations of the Dirac masses $x_k$ are pairwise distinct, we have
\begin{align}
	\mnorm{w_{\V a,\V x}} = \frac1{2\pi}\normu{\V a}.
\end{align}
We also consider the set of nonnegative Radon measures $\Radp$, which are Radon measures $w$ such that $\dotp{w}{\varphi}\geq 0$ for any positive continuous function $\varphi$ on $\TT$. Similarly, we define the set of nonpositive Radon measures as $\Radn$.

\paragraph{Forward operator.}
For any $k \in \Z$, we define $e_k : \T \rightarrow \C$ as $e_k(t) = \mathrm{e}^{\mathrm{i} k t}$.
We consider the measurement operator given by
\begin{align}
	\nuf : w\in\Radon\mapsto (\widehat{w}[-K_c], \ldots, \widehat{w}[0], \ldots \widehat{w}[K_c])\in\CC^{K_c}\times\RR\times\CC^{K_c},
\end{align}
where $K_c\in\NN$ is the cutoff frequency and for all $|k|\leq K_c$, 
\begin{equation}
   \widehat{w}[k] =  \frac{1}{2\pi} \int_0^{2\pi} \ee^{-\ii kt} \dd w(t) = \langle w , e_k \rangle
\end{equation}
is the $k$-th Fourier coefficient of $w$. 
Note that we extended the duality product~\eqref{eq:dualityproduct} to \emph{complex} test functions $\varphi : \R \rightarrow \C$ as $\langle w , \varphi \rangle = \frac{1}{2\pi} \int_{0}^{2\pi} \overline{\varphi(t)} \dd w(t)$.
The operator $\nuf$ is a low-pass operator that only keeps the low-frequency Fourier coefficients.

\paragraph{Optimization problem.}Let $\V y = (y_{-K_c}, \ldots, y_0, \ldots, y_{K_c})\in\CC^{2K_c+1}$ such that\footnote{These requirements on the observations vector $\V y$ come from the fact that otherwise, the equation $\nuf(w)=\V y$ has no solution in $\Radon$, since $w$ is a real Radon measure.}
\begin{align}\label{eq:measurements-vector}
	y_0\in\RR \qandq \forall k\in\{1,\ldots,K_c\}, \ y_{k} = \overline{y_{-k}},
\end{align}
be a given observations vector. We aim to solve the linear inverse problem $\nuf(w) = \V y$. We introduce our terminology for solutions of this problem in the following definition.

\begin{definition}\label{def:representing}
Let $\V y\in\CC^{2K_c+1}$ given as in Equation~\eqref{eq:measurements-vector}. We say that $w\in\Radon$ is a \emph{representing measure for $\V y$}, if $w$ satisfies $\nuf(w) =\V y$.
\end{definition}

This linear inverse problem is ill-posed, as it has infinitely many representing measures. To tackle this issue, we choose to favor sparse solutions. Our approach consists in solving the following optimization problem
\begin{align}\label{eq:BPC}
	\underset{w\in\Radon, \ \nuf(w)=\V y}{\min} \ \mnorm{w}. \tag{BPC}
\end{align}
A solution of~\eqref{eq:BPC} (which is known to exist) therefore has the minimal total-variation norm among all representing measures for $\V y$. As the total-variation norm is an extension of the $\ell^1$ norm, it is known to promote sparse solutions (composed of a sum of Dirac masses). We denote this problem the basis pursuit in the continuum (BPC) as a tribute to its finite-dimensional counterpart, the Basis Pursuit~\cite{chen2001atomic}.

\subsection{Dual Certificates for the basis pursuit in the continuum}\label{subsec:duality}

The analysis of~\eqref{eq:BPC} benefits from the theory of duality for infinite-dimensional convex optimization, as exposed for instance by Ekeland and Temam in~\cite{Ekeland1976convex}. This line of research has proven to be fruitful to study optimization on measure spaces~\cite{de2012exact,candes2014towards,Duval2015exact,Fernandez-Granda2016Super,duval2017sparseI,debarre2020sparsest}. We mostly rely on the concepts and results exposed in~\cite{Duval2015exact,debarre2020sparsest}, but very similar tools can be found elsewhere~\cite{de2012exact,candes2014towards}. Considering the dual problem to~\eqref{eq:BPC} and writing the optimality conditions that link the solutions of both problems\footnote{Note that dual certificates always exist for the (BPC) with Fourier measurements~\cite{Candes2013Super}.} leads to the notion of \emph{dual certificates}, which are continuous functions on $\TT$ satisfying some conditions (see Proposition~\ref{prop:certification} below). In particular, dual certificates enable to \emph{certify} that some $w\in\Radon$ is a solution of~\eqref{eq:BPC} and to localize its support. In the next definition, we introduce some notations that ease the related statements. We recall that a Radon measure $w\in\Radon$ can be uniquely decomposed as $w = w_+ - w_-$, where $w_+$ and $w_-$ are nonnegative measures (Jordan decomposition).

\begin{definition} \label{def:supp-sat}
Let $w\in\Rad$. We define the \emph{signed support} of $w$ as
\begin{align}
	\ssupp w =\supp(w_+) \times \{1\} \cup \supp(w_-) \times \{-1\},
\end{align}
where $\supp(\tilde w)$ is the support of $\tilde w\in\Radon$.

Let $\eta \in \Cc(\TT)$. The positive and negative saturation sets of $\eta$ are given by
\begin{align}
	\satp \eta = \eta^{-1} (\{1\}) \qandq \satn \eta = \eta^{-1} (\{-1\}),
\end{align}
respectively. Finally, we define the \emph{signed saturation set} of $\eta$ as
\begin{align}
	\ssat \eta =\satp{\eta} \times \{1\} \cup \satn{\eta} \times \{-1\}.
\end{align}
\end{definition}

Now, we give two results which help us to characterize the cases of uniqueness for \eqref{eq:BPC}, in our main contribution Theorem~\ref{thm:main}. The next proposition introduces formally the notion of dual certificates for the~\eqref{eq:BPC} problem.

\begin{proposition}\label{prop:certification}
There exists a function $\eta\in\Cc(\TT)$, which is a real trigonometrical polynomial of degree at most $K_c$, satisfying $\normi{\eta}\leq1$ and such that for any solution $w_0\in\Radon$ of~\eqref{eq:BPC}, we have one of the following equivalent conditions:
\begin{itemize}
    \item $\dotp{w_0}{\eta} = \mnorm{w_0}$;
    \item $\ssupp{w_0}\subset\ssat{\eta}$.
\end{itemize}
Such a function $\eta$ is denoted as a \emph{dual certificate}.
\end{proposition}

Proposition~\ref{prop:certification} is stated in some other, nonetheless equivalent, form in~\cite[Proposition 3]{Duval2015exact}, where dual certificates of the optimization problem~\eqref{eq:BPC} are studied\footnote{Duval and Peyré consider more general measurement operators whose image can lie in a Hilbert space and exemplify their results for low-frequency measurements.}. A key role is played by the adjoint operator $\bm{\nu}^* : \CC^{K_c}\times \RR \times \CC^{K_c} \rightarrow \Cc(\TT)$ (denoted by $\Phi^*$ in~\cite{Duval2015exact}), due to the fact that dual certificates must be in $\mbox{Im}(\nuf^*)$. 
In our setup, we have that
\begin{equation} \label{eq:dual-certif-poly-trigo}
	\bm{\nu}^* (c_{-K_c},\ldots, c_{-1}, c_0, c_1, \ldots , c_{K_c}) = \sum_{|k| \leq K_c} c_k e_{-k}. 
\end{equation}
The role of real trigonometric polynomials in Proposition~\ref{prop:certification} is explained by the fact that the vectors $\V c\in \CC^{K_c}\times \RR \times \CC^{K_c}$ involved in our study are Hermitian symmetric because they belong to $\mbox{Im}(\nuf)$. We do not provide a detailed proof of Proposition~\ref{prop:certification}, since it has already been exposed elsewhere. It is for instance done in~\cite[Propositions 1 \& 2]{debarre2020sparsest} in a different setting, but the arguments can be readily adapted.
       
An important consequence of Proposition~\ref{prop:certification} for the study of uniqueness of~\eqref{eq:BPC}, is the following proposition that can also be deduced from~\cite{Duval2015exact} (see also \cite{de2012exact}).

\begin{proposition}\label{prop:certifexistsowhat}
If there exists a nonconstant\footnote{In the more general case studied in~\cite{Duval2015exact}, this corresponds to the nondegeneracy condition of the dual certificate.} dual certificate for the problem~\eqref{eq:BPC}, then it has a unique solution of the form $w = \sum_{k=1}^K a_k \dirac{x_k}$ with $K \leq 2 K_c$ and $a_k\in\RR$, $x_k\in\TT$.
\end{proposition}

For the sake of completeness, a proof of Proposition~\ref{prop:certifexistsowhat} can be found in Appendix~\ref{appendix:proof_certifexistsowhat}.


\section{Toeplitz Characterization of~\eqref{eq:BPC}}\label{sec:mainresults}

In this section, we present our contributions. We recall that all the main notations are introduced in Section~\ref{sec:optim-problem}. In particular, the measurement vector $\V y\in\CC^{K_c}\times\RR\times\CC^{K_c}$ is given in Equation~\eqref{eq:measurements-vector}. Theorem~\ref{thm:equivalencegame} in Section~\ref{sec:criterionBPC} first provides several equivalent conditions which ensure that the solution set of~\eqref{eq:BPC} is solely composed of nonnegative (or nonpositive) measures. This can be proved not to hold when a simple criterion on the coefficient values of $\V y$ is satisfied, see Corollary~\ref{cor:unique-y0-small}. By the second part of Theorem~\ref{thm:equivalencegame}, this leads to the existence of a unique sparse solution composed of at most $2K_c$ Dirac masses, with at least one positive and one negative weight. Theorem~\ref{thm:equivalencegame} paves the way towards our main contribution, Theorem~\ref{thm:main} in Section~\ref{sec:uniqueness-result}, which characterizes the solution set of~\eqref{eq:BPC} from simple condition on the spectrum of the Toeplitz and Hermitian symmetric matrix $\Ty$ introduced in~\eqref{eq:defTy}. In particular, it gives all the cases where uniqueness holds, see Corollary~\ref{cor:cns_uniqueness}.

\subsection{A New Criterion of Uniqueness for~\eqref{eq:BPC}}\label{sec:criterionBPC}

Before stating the main results of this section, let us first prove, in the next lemma, that the total-variation norm of a Radon measure upper-bounds its Fourier coefficients. Lemma~\ref{lemma:RF} also provides elementary characterizations for nonnegative Radon measures that we use in Theorem~\ref{thm:equivalencegame}.

\begin{lemma} \label{lemma:RF}
Let $w \in \Rad$. Then, 
\begin{enumerate}
\item For any $k \in \Z$, we have $|\widehat{w}[k]| \leq \mnorm{w}$;
\item $w \in \Radp$ if and only if $\mnorm{w} = \widehat{w}[0]$;
\end{enumerate}
\end{lemma}

\begin{proof}
Let $k \geq 0$. Any $w \in \Rad$ can be uniquely decomposed as $w = w_+ - w_-$ where $w_+$, $w_- \in \Radp$  (Jordan decomposition).
We then observe that 
\begin{align}
    |\widehat{w}[k]| 
    &= |\langle w , e_k \rangle| 
    = | \langle w_+  , e_k \rangle - \langle w_-  , e_k \rangle |  
    \leq
    | \langle w_+  , e_k \rangle |
    + 
    | \langle w_-  , e_k \rangle | \nonumber \\
    & \leq 
   \frac{1}{2\pi} \int_{\T} | e_k(t) | \mathrm{d}w_+ (t) 
    + 
   \frac{1}{2\pi} \int_{\T} | e_k(t) | \mathrm{d}w_- (t) 
    = \langle w_+ , 1 \rangle +  \langle w_- , 1 \rangle \nonumber \\
    &= \| w \|_{\mathcal{M}}. 
\end{align}
For 2., we observe that $\lVert w \rVert_{\Mm} = \lVert w_+ \rVert_{\Mm}  + \lVert w_- \rVert_{\Mm} = \langle w_+ , 1 \rangle +   \langle w_- , 1 \rangle = \widehat{w}_+[0] + \widehat{w}_-[0]$. Then,
\begin{equation}
	w \in \Radp \ \Leftrightarrow  \  w_- = 0 \ \Leftrightarrow  \  \widehat{w}_-[0] = 0 \ \Leftrightarrow  \ \lVert w \rVert_{\Mm} = \widehat{w}_+[0] = \widehat{w}[0],
\end{equation}
which concludes the proof.
\end{proof}

Next, we provide a lower bound on the minimal value of~\eqref{eq:BPC}.
\begin{lemma} \label{lemma:minwmaxy}
We have that
\begin{equation} \label{eq:minwmaxy}
	\underset{w \in \Rad, \ \nuf(w) = \bm{y}}{\min} \quad 
	\mnorm{w}\geq\max_{0\leq k \leq K_c} |y_k|. 
\end{equation}
\end{lemma}
\begin{proof}
We know that~\eqref{eq:BPC} has at least one solution $w_0\in\Radon$ that reaches its minimum value. Then, according to Lemma~\ref{lemma:RF}, we have $\mnorm{w_0} \geq |\widehat{w_0}[k]| = |y_k|$ for all $-K_c\leq k \leq K_c$, which yields~\eqref{eq:minwmaxy}.
\end{proof}

We can now state the main result of this section.
\begin{theorem}\label{thm:equivalencegame}
Let $\V y$ be as in~\eqref{eq:measurements-vector}. We define $\varepsilon = \varepsilon(y_0) \in \{+, -\}$ such that $\varepsilon$ is $+$ if $y_0\geq 0$, and $\varepsilon$ is $-$ if $y_0 < 0$. Then, the following conditions are equivalent:
\begin{enumerate}
\item\label{enum:1} there exists $w_0 \in \Radeps$ such that $\nuf(w_0) = \bm{y}$;
\item\label{enum:3} the solution set of~\eqref{eq:BPC} is $\{ w \in \Radeps: \ \nuf(w) = \bm{y} \}$;
\item\label{enum:2} we have the following equality:
\begin{equation} \label{eq:minisyo}
	|y_0| =  \underset{w \in \Radon, \ \nuf(w) = \bm{y}}{\min} \quad 
	\mnorm{w}.
\end{equation}

\end{enumerate}
If the above equivalence is not satisfied, then~\eqref{eq:BPC} has a unique solution, composed of at most $2K_c$ Dirac masses, with at least one positive and one negative weight.
\end{theorem}

\begin{proof}
Let us first prove that the items~\ref{enum:1},~\ref{enum:3},~\ref{enum:2} are equivalent. We assume that $y_0\geq0$. The case $y_0 < 0$ is similar.
\subparagraph{\ref{enum:1}$\Rightarrow$\ref{enum:2}.}Set $m = \underset{w \in \Radon, \ \nuf(w) = \bm{y}}{\min} \mnorm{w}$. Let $w_0\in\Radp$ such that $\nuf(w_0) = \bm y$. By Lemma~\ref{lemma:RF}, we have, for any $w \in \Radon$, the equivalence
\begin{align}\label{eq:caractwpositive}
	w\in\Radp\Longleftrightarrow\mnorm{w} = \widehat{w}[0].
\end{align}
As a result, we get that $m \leq \mnorm{w_0} = \widehat{w_0}[0] = y_0$. Finally, from Lemma~\ref{lemma:minwmaxy}, we have that $m \geq y_0$, which yields $m = y_0$\footnote{Note that we also proved that $w_0$ is solution of~\eqref{eq:BPC}.}.

\vspace{-0.2cm}

\subparagraph{\ref{enum:2}$\Rightarrow$\ref{enum:3}.}
Let $w_0\in\Radon$ be a solution of~\eqref{eq:BPC} which is equivalent to $\underset{w \in \Radon, \ \nuf(w) = \bm{y}}{\min} \quad \mnorm{w} =  \mnorm{w_0}$. Thus, by~\ref{enum:2}, this last equality can be equivalently rewritten as $y_0 = \mnorm{w_0}$, i.e. $\widehat{w_0}[0] = \mnorm{w_0}$. By Lemma~\ref{lemma:RF}, it is equivalent to $w_0\in\Radp$. As $w_0$ satisfies the constraints $\nuf(w_0) = \bm{y}$, we finally obtain that the solution set of~\eqref{eq:BPC} is equal to $\{ w \in \Radp: \ \nuf(w) = \bm{y} \}$.

\vspace{-0.2cm}

\subparagraph{\ref{enum:3}$\Rightarrow$\ref{enum:1}.}This is a direct consequence of the fact that the solution set of~\eqref{eq:BPC} is non-empty. \\

\vspace{-0.2cm}

Next, if the equivalence is not satisfied, this implies that there is no solution that is a nonnegative or a nonpositive measure. Let $w_0\in\Rad$ be a solution and $\eta$ a dual certificate for problem~\eqref{eq:BPC}, as given by Proposition~\ref{prop:certification}. Using Proposition~\ref{prop:certifexistsowhat}, we know that if the dual certificate is nonconstant, then the solution of~\eqref{eq:BPC} is unique and composed of at most $2K_c$ Dirac masses, with at least one negative and one positive weight. Let us assume by contradiction that $\eta$ is constant. As $w_0\neq 0$, we have $\ssupp{w_0}\neq\varnothing$, hence $\ssat{\eta}\neq\varnothing$. Since $\eta$ is constant, we thus have $\ssat{\eta} = \TT\times\{1\}$ or $\ssat{\eta} = \TT\times\{-1\}$, hence $w_0$ is nonnegative or nonpositive. This contradicts our initial assumption, which concludes the proof.
\end{proof}

\begin{remark}\label{rem:positivity_sparsity}
Theorem~\ref{thm:equivalencegame} shows that there exists a nonnegative representing measure for $\V y$ if and only if the solution set of~\eqref{eq:BPC} is composed of all the nonnegative representing measures for $\V y$. This suggests that minimizing the TV norm plays no role in this context and can be replaced by a simple nonnegativity constraint. This is consistent with recent results~\cite{pouchol2020ml,pouchol2020linear,eftekhari2021sparse,eftekhari2021stable} which proved, in different setups, that when the observations $\V y$ are generated by a nonnegative \emph{sparse} measure, then replacing the TV norm with a nonnegativity constraint is enough to uniquely recover the input measure, provided that the number of measurements is sufficient. It is worth noting that, contrary to the cited works, we do not assume in Theorem~\ref{thm:equivalencegame} that $\V y$ is generated by a \emph{sparse} measure that we would like to recover.

More importantly, Theorem~\ref{thm:equivalencegame} proves that when the equivalence is not satisfied, then we must be in a case of uniqueness, which is, to the best of our knowledge, a new result.
\end{remark}

From Theorem~\ref{thm:equivalencegame} and Lemma~\ref{lemma:minwmaxy}, we can deduce a simple criterion on the magnitude of the coefficients of $\V y$ which ensures that the solution of~\eqref{eq:BPC} is unique. This criterion appears to be practically fruitful since uniqueness follows trivially for a large class of observation vectors $\bm{y}$.
        
\begin{corollary} \label{cor:unique-y0-small}
If $|y_0| < |y_{k_0}|$, for some $k_0 \neq 0$, then~\eqref{eq:BPC} has a unique solution, composed of at most $2K_c$ Dirac masses, with at least one positive and one negative weight.
\end{corollary}

\begin{proof}
By Lemma~\ref{lemma:minwmaxy}, we have $\min_{w\in\Radon, \ \nuf(w)=\bm{y}} \mnorm{w} \geq |y_{k_0}| > |y_0|$. This implies that the equivalent conditions of Theorem~\ref{thm:equivalencegame} do not hold, as Item~\ref{enum:2} is not valid. We are thus in the uniqueness scenario of Theorem~\ref{thm:equivalencegame}.
\end{proof}

\subsection[\texorpdfstring{tex}{pdfbookmark}]{Characterization of Solutions of~\eqref{eq:BPC} Using $\Ty$}\label{sec:uniqueness-result}

In this section, we prove, in Theorem~\ref{thm:main}, that the sign of the solutions of~\eqref{eq:BPC} is directly related to spectral properties of the matrix $\Ty$ formed from the observation vector $\V y$ defined as
\begin{align}\label{eq:Ty_recall}
	\Ty = 2\pi\begin{pmatrix} 
	y_0 & y_1 & y_2 & \ldots & y_{K_c} \\
	y_{-1} & y_0 & y_1 & \ldots & y_{K_c-1} \\
	y_{-2} & y_{-1} & y_0 & \ldots & y_{K_c-2} \\
	\vdots &  & & \ddots  &  \\
	y_{-{K_c}} & y_{-K_c+1} & y_{-K_c+2} & \ldots & y_{0} \\
\end{pmatrix}.
\end{align}
More precisely, as $\Ty$ is Hermitian symmetric, one of the following three statements must hold:
\begin{itemize}
\item $\Ty$ has at least one negative and one positive eigenvalue;
\item $\Ty$ is positive or negative semi-definite and rank deficient;
\item $\Ty$ is positive or negative-definite.
\end{itemize}
In Theorem~\ref{thm:main}, we prove that each of these scenarios leads to different properties of the solution set of~\eqref{eq:BPC}, all expressed in terms of the nature of the sign of the solutions and of uniqueness. To this end, we first characterize the existence of a nonnegative representing measure for $\V y$ in terms of positive semi-definiteness of the matrix $\Ty$ in Proposition~\ref{prop:curto}. It can be obtained from~\cite[Theorem 6.12]{curto1991recursiveness}, but we include for the sake of completeness. 

\begin{proposition}\label{prop:curto}
The two following conditions are equivalent:
\begin{enumerate}
\item\label{enum-prop:1} there exists $w_0 \in \Radp$ such that $\nuf(w_0) = \bm{y}$;
\item\label{enum-prop:2} the Hermitian matrix $\Ty \in \C^{(K_c+1)\times (K_c+1)}$, defined in~\eqref{eq:Ty_recall}, is positive semi-definite.
\end{enumerate}
\end{proposition}

\begin{proof}
\
\vspace{-0.4cm}

\subparagraph{\ref{enum-prop:1}$\Rightarrow$\ref{enum-prop:2}.}Let $w_0$ be a nonnegative representing measure for $\V y$. Consequently, by Proposition~\ref{prop:herglotz} in Appendix~\ref{appendix:toolbox}, we have that $\sum_{k,\ell \in \Z} \widehat{w_0}[k-\ell] z_k \overline{z}_\ell \geq 0$, for any complex sequence $(z_k)_{k\in\Z}$ with finitely many nonzero terms. In particular, restricting to sequences such that $z_k = 0$ for $k < 0$ and $k> K_c$, we have that
\begin{equation} \label{eq:equation}
	\sum_{k,\ell=0}^{K_c} \widehat{w_0}[k-\ell] z_k \overline{z}_\ell  \geq 0. 
\end{equation}
Since $\widehat{w_0}[k-\ell]=y_{k-\ell}$ for all $k,\ell\in\{0,\ldots,K_c\}$, we deduce that $\Ty$ is positive semi-definite.

\vspace{-0.2cm}

\subparagraph{\ref{enum-prop:2}$\Rightarrow$\ref{enum-prop:1}.}We denote $\Pp_{K_c}(\T) = \Span(e_{-K_c},e_{-K_c + 1},\ldots, e_{K_c})$ (where we recall that $e_k(t) = \mathrm{e}^{\mathrm{i} k t}$) the space of real trigonometric polynomials of degree at most $K_c$. Let us consider the linear mapping $ \Phi : \Pp_{K_c}(\T) \rightarrow \R$ such that for all $|k|\leq K_c$, $\Phi(e_k) = y_k$. Then, $\Phi$ must be positive. Indeed, let $p \in \Pp_{K_c}(\T)$ such that $p \geq 0$. According to Proposition~\ref{prop:fejer} in Appendix~\ref{appendix:toolbox}, $p$ can be written as $p = |q|^2$ for some complex trigonometric polynomial $q= \sum_{k=0}^{K_c} z_k e_k$ with $z_k \in \C$. This implies that $p = \sum_{k, \ell=0}^{K_c} z_k \overline{z}_\ell e_{k-\ell}$ and therefore that
\begin{equation}
	\Phi(p) = \sum_{k, \ell=0}^{K_c} y_{k-\ell}  z_k \overline{z}_\ell = \frac1{2\pi}\dotp{\Ty \V z}{\V z} \geq 0 \qwhereq \V z = (z_0,\ldots,z_{K_c})\in\CC^{K_c+1}.
\end{equation}
Then, according to Proposition~\ref{prop:extensionpositive} in Appendix~\ref{appendix:toolbox}, $\Phi$ can be extended to a positive, linear, and continuous functional from $\Cc(\T)$ to $\R$. This implies by the Riesz-Markov theorem~\cite{gray1984shaping} that there exists $w_0\in\Radp$ such that $\Phi = w_0$. Moreover, by construction of $\Phi$, $w_0$ satisfies $\nuf(w_0)= \bm{y}$. Consequently, Item~\ref{enum:1} is proved.
\end{proof}

\begin{remark}
Proposition~\ref{prop:curto} can readily be adapted to the nonpositive case, since there exists a nonpositive representing measure for $\V y$ if and only if $\Ty$ is negative semi-definite.
\end{remark}

Proposition~\ref{prop:curto} gives another equivalent condition to Item~\ref{enum:1} of Theorem~\ref{thm:equivalencegame} which involves the matrix $\Ty$. Building on this result, the statement on the nature of the solutions of~\eqref{eq:BPC} can be refined by leveraging the fact that $\Ty$ is a Toeplitz matrix. Indeed, it is well known that a positive semi-definite and rank deficient Toeplitz matrix can be uniquely decomposed as the product $\bm{\mathrm{V}}\bm{\mathrm{D}}\bm{\mathrm{V}}^*$, where $\bm{\mathrm{V}}$ is a Vandermonde matrix whose columns are given by complex exponentials, $\bm{\mathrm{D}}$ is a positive-definite diagonal matrix whose size is the rank of the starting Toeplitz matrix and $V^*$ is the conjugate transpose of $V$. This is the Pisarenko decomposition~\cite{caratheodory1911zusammenhang,pisarenko1973retrieval,yang2018frequency}, also known as the Carathéodory-Fejér-Pisarenko (CFP) decomposition. This result is recalled in Appendix~\ref{appendix:CFP}, for the sake of completeness. The next lemma is then the last missing piece towards Theorem~\ref{thm:main}; it relates this decomposition of the Toeplitz matrix $\Ty$ to the existence of a sparse representing measure for $\V y$.

\begin{lemma}\label{lem:toeplitzdecomp}
For all $0\leq K\leq K_c+1$, $\bm a = (a_1,\ldots,a_K)\in(\RR_{>0})^K$ and $\bm x = (x_1,\ldots,x_K)\in\TT^K$ with pairwise distinct entries, the following statements are equivalent:
\begin{enumerate}[label=\alph*)]
\item\label{statement:y} $\bm y = \bm\nu(w_{\bm a,\bm x})$ with $w_{\bm a,\bm x} = \sum_{k=1}^K a_k \dirac{x_k}$;
\item\label{statement:Ty} $\Ty = \Vx \Da \Vx^*$ where $\Vx\in\CC^{(K_c+1)\times K}$ is the Vandermonde matrix whose $k$-th column is given by $\bm{e_{K_c}}(x_k) = (1 \ \ee^{\ii x_k} \ \cdots \ \ee^{\ii K_c x_k})^T\in\RR\times\CC^{K_c}$ and $\Da\in\RR^{K\times K}$ the diagonal matrix with the entries of $\bm a$ on the diagonal.
\end{enumerate}
\end{lemma}

\begin{proof}
Let $K\leq K_c+1$, $\bm a\in(\RR_{>0})^K$, and $\bm x\in\TT^K$ be as in the statement of the lemma.

First, suppose that~\ref{statement:y} holds. Then, for any $0\leq m\leq K_c$ and $0\leq n\leq K_c-m$, we have
\begin{align}\label{eq:y_m}
	y_m &= \frac1{2\pi}\sum_{k=1}^K a_k \ee^{-\ii m x_k} = \frac1{2\pi}\sum_{k=1}^K a_k \ee^{\ii n x_k}\ee^{-\ii (m + n) x_k} = \frac1{2\pi}\dotp{(\ee^{\ii n x_k})_{1\leq k\leq K}}{(a_k \ee^{\ii (m + n)x_k})_{1\leq k\leq K}}.
\end{align}
We notice that $(\ee^{\ii n x_k})_{1\leq k\leq K}$ is the $(n+1)$-th row of the matrix $\Vx$ and $(a_k \ee^{-\ii (m + n)x_k})_{1\leq k\leq K}$ the $(m+n+1)$-th column of $\Da\Vx^*$. Therefore, $2\pi y_m$ is the $(n+1, m+n+1)$ entry of the matrix $\Vx\Da\Vx^*$. Hence, the elements of the $m$-th upper diagonal of $\Vx\Da\Vx^*$ are $2\pi y_m$, and since $\Vx\Da\Vx^*$ is Hermitian symmetric, we get that $\Vx\Da\Vx^*=\Ty$.

Conversely suppose that~\ref{statement:Ty} holds. Then, the first line of $\Ty$ gives the vector $2\pi \bm y$ and as demonstrated in~\eqref{eq:y_m} the first line of $\Vx\Da\Vx^*$ gives $2\pi\bm\nu(w_{\bm a,\bm x})$, which implies that $\V y = \nuf(w_{\V a,\V x})$.
\end{proof}

\begin{remark}
\label{rem:sign-lem-equi}
The statements of Lemma~\ref{lem:toeplitzdecomp} can once again readily be adapted to the case where the weights $\bm a$ are negative. In this case, the Toeplitz matrix $\Ty$ is negative semi-definite.
\end{remark}

We can now state our main result, which relies on Theorem~\ref{thm:equivalencegame}, Proposition~\ref{prop:curto}, Proposition~\ref{prop:CFP}, and Lemma~\ref{lem:toeplitzdecomp}.

\begin{theorem}\label{thm:main}
The solution set of~\eqref{eq:BPC} can be characterized as follows:
\begin{enumerate}[label=\alph*)]
\item\label{cond:Ty_pos-and-neg} If $\Ty$ has at least one negative and one positive eigenvalue, then~\eqref{eq:BPC} has a unique solution composed of at most $2K_c$ Dirac masses, with at least one positive and one negative weight;
\item\label{cond:Ty_semi_pos-and-neg} If $\Ty$ is positive, respectively negative, semi-definite and $\rank(\Ty)< K_c + 1$, then~\eqref{eq:BPC} has a unique solution composed of exactly $\rank(\Ty)$ positive, respectively negative, Dirac masses;
\item\label{cond:Ty_pos-or-neg} If $\Ty$ is positive, respectively negative, definite (which implies $\rank(\Ty)=K_c+1$), then~\eqref{eq:BPC} has infinitely many solutions. Moreover, there is no solution that is sum of less than $K_c+1$ Dirac masses. Finally, there are uncountably many solutions composed of exactly $K_c+1$ Dirac masses with positive, respectively negative, weights.
\end{enumerate}
\end{theorem}

\begin{proof}
If~\ref{cond:Ty_pos-and-neg} holds, then by Proposition~\ref{prop:curto} there is no nonnegative or nonpositive representing measure for $\V y$. Hence, by Theorem~\ref{thm:equivalencegame},~\eqref{eq:BPC} has a unique solution composed of at most $2K_c$ Dirac masses, with at least one positive and one negative weight.

Suppose that the assumptions of~\ref{cond:Ty_semi_pos-and-neg} are satisfied. Without loss of generality, we can assume that $\Ty$ is positive semi-definite. Then by the CFP decomposition, which is recalled in Proposition~\ref{prop:CFP} in Appendix~\ref{appendix:CFP}, there exist unique sets $\{x_1,\ldots,x_K\}\subset\TT$ (the $x_k$ are pairwise distinct) and $\{a_1,\ldots,a_K\}\subset\RR_{>0}$ with $K=\rank(\Ty)<K_c+1$ such that
\begin{align}\label{eq:decompTy}
	\Ty = \Vx \Da \Vx^*,
\end{align}
with $\bm x = (x_1,\ldots, x_K)$, $\bm a = (a_1,\ldots, a_K)$, $\Vx\in\CC^{(K_c+1)\times K}$ the Vandermonde matrix whose $k$-th column is given by $\bm{e_{K_c}}(x_k) = (1 \ \ee^{\ii x_k} \ \cdots \ \ee^{\ii K_c x_k})^T\in\CC^{K_c+1}$ and $\Da\in\RR^{K\times K}$ is a diagonal matrix with $\bm a$ on the diagonal. By Lemma~\ref{lem:toeplitzdecomp}, Equation~\eqref{eq:decompTy} is equivalent to
\begin{align}
	\bm y = \bm\nu(w_{\bm a, \bm x}) \qwithq w_{\bm a, \bm x} = \sum_{k=1}^K a_k\dirac{x_k}.
\end{align}
Then, according to Theorem~\ref{thm:equivalencegame} (equivalence between Items~\ref{enum:1} and~\ref{enum:3}), the solution set of~\eqref{eq:BPC} is $\{w\in\Radp : \nuf(w) = \bm y\}$. Therefore, $w_{\bm a,\bm x}$ is a solution of~\eqref{eq:BPC}, which is composed of $K<K_c+1$ positive Dirac masses. By Lemma~\ref{lem:toeplitzdecomp}, the uniqueness of the CFP decomposition, and Item~\ref{enum:3} of Theorem~\ref{thm:equivalencegame}, it is the unique solution of~\eqref{eq:BPC} with less than $K_c + 1$ Dirac masses. By~\cite[Theorem 2.1]{de2012exact}, it is the unique solution altogether, since it is possible to build a nonconstant dual certificate such that its signed saturation set is exactly $\ssupp{w_{\bm a, \bm x}}$, which is a sufficient condition of uniqueness by Proposition~\ref{prop:certifexistsowhat}.

Finally, suppose that the assumption of~\ref{cond:Ty_pos-or-neg} is satisfied. Without loss of generality, we can assume that $\Ty$ is positive-definite. Once again by Proposition~\ref{prop:CFP}, there are uncountably many decompositions $\Ty=\Vx\Da\Vx^*$ with $\bm x = (x_1,\ldots,x_{K_c +1})\in\TT^{K_c+1}$ whose elements are pairwise distinct and $\bm a=(a_1,\ldots,a_{K_c+1})\in(\RR_{>0})^{K_c+1}$. By Lemma~\ref{lem:toeplitzdecomp} and Theorem~\ref{thm:equivalencegame} (equivalence between Items~\ref{enum:1} and~\ref{enum:3}), this implies that for all these $(\bm a, \bm x)$ pairs, $w_{\bm a,\bm x}$ is a solution of~\eqref{eq:BPC}\footnote{\label{footnote:Kc}As stated in Proposition~\ref{prop:CFP}, one can arbitrarily choose one of the $x_k$ in $\TT$ in the decomposition $\Ty=\Vx\Da\Vx^*$, which proves that $t\in\TT\mapsto 1$ is the unique dual certificate in this setup. Indeed any $(x_k,1)\in\TT\times\{\pm 1\}$ must be in the signed saturation set of any dual certificate.}. As a result, there are uncountably many solutions consisting of $K_c+1$ Dirac masses with positive weights. There can be no solution consisting of less than $K_c+1$ Dirac masses, since they would necessarily have positive weights and once again, by~\cite[Theorem 2.1]{de2012exact}, that solution would be the unique one, which we have proved to be false.
\end{proof}

\begin{remark}\label{rem:curto}
Item~\ref{cond:Ty_semi_pos-and-neg} of Theorem~\ref{thm:main} can be deduced from well-grounded results in the literature. Indeed, it is known from a classical result~\cite[Theorem 6.12]{curto1991recursiveness} in the field of  moment problems, that $\V y$ has a unique nonnegative representing measure consisting of $\rank(\Ty)$ Dirac masses if and only if $\Ty$ is positive semi-definite and rank deficient (\ie $\rank(\Ty)\leq K_c$). Next, this measure is the unique solution of~\eqref{eq:BPC}, since by~\cite[Theorem 2.1]{de2012exact}, $\V y$ is generated from a nonnegative measure composed of less than $K_c+1$ Dirac masses. This can also be proved by the equivalence between Items~\ref{enum:1} and~\ref{enum:3} of Theorem~\ref{thm:equivalencegame}. 
We also again recover the fact, mentioned in Remark~\ref{rem:positivity_sparsity}, that the TV norm plays no role in this context of nonnegativity. Concerning Item~\ref{cond:Ty_pos-or-neg} when $\Ty$ is full rank,~\cite[Theorem 6.12]{curto1991recursiveness} states that there are infinitely many nonnegative representing measures composed of $K_c+1$ Dirac masses. However, Theorem~\ref{thm:equivalencegame} is, to the best of our knowledge, the first known result which proves that they are all solutions of~\eqref{eq:BPC}.
\end{remark}

\begin{remark}
Theorem~\ref{thm:main} provides a sharper (and tight\footnote{Consider the measure $w_0 = \sum_{k=0}^{2K_c-1} (-1)^k \dirac{\frac{2\pi k}{2K_c}}$, the measurements $\V y=\nuf(w_0)$, and the function $\eta\in\Cc(\TT)$ defined as $\eta = t\mapsto \cos(K_c t)$. Then, $\eta$ is a real trigonometric polynomial of degree $K_c$ and $\ssupp{w_0}\subset\ssat{\eta}$, which implies by Proposition~\ref{prop:certification} that $w_0$ is a solution of~\eqref{eq:BPC}.}) upperbound, $2K_c$, on the number of Dirac masses of a sparse solution for~\eqref{eq:BPC} than Representer Theorems~\cite{fageot2020tv} which give the upperbound $2K_c+1$. This improved bound arises when the assumptions of~\ref{cond:Ty_pos-and-neg} hold. Note that in this context, the sparsity of the unique solution is no longer equal to the rank of $\Ty$ since the CFP decomposition is valid only when $\Ty$ is positive (or negative) semi-definite. This is confirmed by the fact that one can construct examples of a solution of~\eqref{eq:BPC} consisting of more than $K_c+1$ Dirac masses\footnotemark[8], while the rank of $\Ty$ is bounded by $K_c+1$.
\end{remark}

From Theorem~\ref{thm:main}, we readily deduce all the situations of uniqueness for~\eqref{eq:BPC}. They are summarized in the next corollary.

\begin{corollary}\label{cor:cns_uniqueness}
The problem~\eqref{eq:BPC} has a unique solution if and only if $\Ty$ is neither positive nor negative definite.
\end{corollary}


\section{Conclusion}
\label{sec:conclusion}

This paper deals with the linear continuous-domain inverse problem~\eqref{eq:BPC}, where the goal is to recover a periodic Radon measure from its low-frequency Fourier coefficients $\V y$ using a sparsity prior. 
We studied the question of uniqueness without making any assumptions of the ground-truth signal that generated $\V y$. In this context, we proved that it can be characterized from simple conditions on the spectrum of the Toeplitz matrix $\Ty$. We also demonstrated that this matrix contains information on the form of the sparse solutions of~\eqref{eq:BPC}, namely their number of Dirac masses, or a bound thereon, and the signs of their weights.

\section*{Acknowledgments}
The authors are grateful to Christian Remling for his help regarding the extension of positive linear functionals in Proposition \ref{prop:extensionpositive}.
They also thank Shayan Aziznejad, Adrian Jarret, Matthieu Simeoni, and Michael Unser for interesting discussions.
Julien Fageot is supported by the Swiss National Science Foundation (SNSF) under Grants P2ELP2\_181759 and P400P2\_194364. The work of Thomas Debarre is supported by the SNSF under Grant 200020\_184646 / 1.


\appendix 

\appendix

\section{Proof of Proposition~\ref{prop:certifexistsowhat}}\label{appendix:proof_certifexistsowhat}

\begin{proof}

Let $\eta$ be a nonconstant dual certificate for~\eqref{eq:BPC} and $w_0\in\Radon$ a solution. Firstly, $\eta$ is a trigonometric polynomial of degree at most $K_c$ and so is its derivative $\eta'$. Since $\eta$ is nonconstant, $\eta'$ has at most $2K_c$ roots~\cite[p. 150]{powell1981approximation}. By Proposition~\ref{prop:certification}, we have $\ssupp{w_0}\subset\ssat{\eta}$, hence any point in the support of $w_0$ is a root of $\eta'$. Consequently, $w_0$ is composed of at most $2K_c$ Dirac masses. Let $\bm{\tau}=(\tau_1,\ldots,\tau_P)\in\TT^P$ be the pairwise distinct roots of $\eta'$, with $P\leq 2K_c$. Then, we have $w_0 = \sum_{p=1}^P a_p \dirac{\tau_p}$, with $\bm a = (a_1,\ldots,a_P)\in\RR^P$ (note that some weights may be equal to $0$). Moreover, any other solution of~\eqref{eq:BPC} must be of the form $w_{\bm{\tilde a}, \bm\tau} = \sum_{p=1}^P \tilde a_p \dirac{\tau_p}$, with $\bm{\tilde a} = (\tilde a_1,\ldots,\tilde a_P)\in\RR^P$ (once again some weights may be equal to $0$), where
\begin{align}\label{eq:weights}
\nuf(w_0)=\V y = \nuf(w_{\bm{\tilde a}, \bm\tau}).
\end{align}
Consider the matrix
\begin{equation}
	\bm{\mathrm{M}_{\bm\tau}} = \frac1{2\pi} \begin{bmatrix} 
	\ee^{\ii K_c \tau_1} & \ldots  & \ee^{\ii K_c \tau_P} \\
 \vdots & \ddots  &  \\
	\ee^{\ii   \tau_1} & \ldots  & \ee^{\ii \tau_P} \\
 1 & \ldots  & 1\\
	\ee^{- \ii  \tau_1} & \ldots  & \ee^{- \ii  \tau_P} \\
 \vdots & \ddots  &  \\
	\ee^{- \ii K_c \tau_1} & \ldots  & \ee^{- \ii K_c \tau_P} \\
	\end{bmatrix} = \frac1{2\pi} \left[\ee^{\ii k \tau_p}\right]^{1\leq p \leq P}_{-K_c\leq k \leq K_c} \in \C^{(2K_c+1)\times P}.
\end{equation}
By definition of $\nuf$ and by Equation~\eqref{eq:weights}, we get that $\nuf(w_0 - w_{\bm{\tilde a}, \bm\tau}) = \bm{\mathrm{M}_{\bm\tau}}(\bm a - \bm{\tilde{a}}) = \bm 0$. The matrix $\bm{\mathrm{M}_{\bm\tau}}$ is a Vandermonde-type matrix, which is therefore of full rank $P$, since $P\leq 2K_c$ and $\tau_1,\ldots,\tau_P$ are pairwise distinct. Hence the nullspace of $\bm{\mathrm{M}_{\bm\tau}}$ is trivial and $\bm a = \bm{\tilde a}$, which prove the uniqueness of the solution $w_0$.
\end{proof}

\section{Trigonometric Toolbox}\label{appendix:toolbox}

This section is dedicated to known theoretical results (or easily deducible therefrom) that play a crucial role in our contributions in this paper.

A sequence $(a_k)_{k\in \Z}$ of complex numbers is \emph{positive definite} if $a_0 \in \RR_+$, $a_{-k} = \overline{a_k}$ for any $k \geq 0$, and for any sequence $(z_k)_{k\in\Z}$ of complex numbers with finitely many nonzero terms, we have
\begin{equation} \label{eq:positivedef}
	\sum_{k, \ell\in \Z}  a_{k-\ell}  z_k \overline{z_\ell} \geq 0. 
\end{equation}
        
\begin{proposition}[Herglotz Theorem]  \label{prop:herglotz}
A sequence $(a_k)_{k\in \Z}$ is positive-definite  if and only if there exists a nonnegative measure $w \in \Radp$ such that $\widehat{w}[k] = a_k$ for all $k\in \Z$.
\end{proposition}

This theorem was obtained by Herglotz in~\cite{herglotz1911uber}. For a modern exposition, we refer to~\cite[Theorem 7.6]{katznelson2004introduction}. The Herglotz theorem is an application of the Bochner theorem, which characterizes the Fourier transform of probability measures on locally Abelian groups $G$~\cite{rudin1962fourier}, here with $G = \T$.

We recall the definition of the complex sinusoid functions $e_k(t) = \mathrm{e}^{\mathrm{i} k t}
$ for $t \in \T$ and $k \in \Z$. 
For $K \geq 0$, we denote by $\mathcal{P}_K(\T)$ the set of real trigonometric polynomial of degree at most $K$; \emph{i.e.}, functions of the form $p= \sum_{|k|\leq K} c_k e_k$ such that $c_0 \in \R$ and $c_{-k} = \overline{c_k} \in \C$ for any $1 \leq k \leq K$.

\begin{proposition}[Fej\'er--Riesz Theorem]\label{prop:fejer}
Let $p = \sum_{|k|\leq K} c_k e_k \in \mathcal{P}_K(\T)$ be a positive trigonometric polynomial of degree $K \geq 0$. Then, there exists a complex trigonometric polynomial $q= \sum_{k=0}^{K_c} z_k e_k$ such that $p = |q|^2$.
\end{proposition}

The Fej\'er--Riesz theorem was conjectured by Fej\'er~\cite{fejer1916trigonometrische} and shown by  Riesz~\cite{riesz1916probelm}. See~\cite[p. 26]{simon2005orthogonal} for a recent exposition of this classical result. 
The next proposition deals with the extension of positive linear functionals from trigonometric polynomials to the space of continuous functions.

\begin{proposition} \label{prop:extensionpositive}
Let $K \geq 0$. Let $\Phi : \Pp_K(\T) \rightarrow \R$ be a linear and positive functional (\emph{i.e.}, $\Phi(p) \geq 0$ for any $p \geq 0$). Then, there exists an extension $\Phi : \Cc(\T) \rightarrow \R$ which is still linear and positive. Moreover, any such extension is continuous on $(\Cc(\T), \lVert \cdot \rVert_\infty)$.
\end{proposition}
\begin{proof}
Let $E$ be an ordered topological vector space, $C$ its positive cone, and $M \subset E$. By~\cite[Corollary 2 p. 227]{schaefer1971locally}, if $C \cap M$ contains an interior point of $C$, then any continuous, positive, and linear form over $M$ can be extended as a continuous, positive, and linear form over $E$.

We apply this result to $E = \Cc(\T)$, whose positive cone is the space of positive continuous functions $\Cc_+(\T)$, and to $M = \mathcal{P}_K(\T)$. Then, $C \cap M = \Cc_+(\T) \cap \mathcal{P}_K(\T)$ contains the constant function $p = 1$, which is an interior point of $\Cc_+(\T)$ since $\{ f \in \Cc(\T), \ \lVert f - 1 \rVert_\infty < \frac{1}{2} \} \subset \Cc_+(\T)$.

In our case, $\Phi$ is continuous over $(\mathcal{P}_K(\T), \lVert \cdot \rVert_\infty)$, since it is a linear functional over a finite-dimensional space. Hence, $\Phi$ is continuous, positive, and linear, and admits the desired extension.
\end{proof}

\section{The Carathéodory-Fejér-Pisarenko Decomposition}\label{appendix:CFP}

In this appendix, we give, for the sake of completeness, the Carathéodory-Fejér-Pisarenko decomposition which plays a major role in our main contribution Theorem~\ref{thm:main}. Proposition~\ref{prop:CFP}, is a transcription from~\cite[Theorem 1]{yang2018frequency}, up to slight adaptations of notations and reformulations. Before stating the result, let us define the column matrix $\bm{e}_K(x) = (1 \ \ee^{\ii x} \ \cdots \  \ee^{\ii K x})^T$, for all $x\in\TT$ and $K\in\NN$.

\begin{proposition}\label{prop:CFP}
Let $K\in\NN$, $\bm{\mathrm{T}}\in\CC^{(K+1)\times (K+1)}$ be a Hermitian Toeplitz matrix, and $r=\rank(\bm{\mathrm{T}})$. Then, the two following conditions are equivalent:
\begin{enumerate}
    \item $\bm{\mathrm{T}}$ is positive semi-definite;
    \item\label{item:decomp} $\bm{\mathrm{T}} = \Vx\Da\Vx^*$, where $\bm x = (x_1,\ldots,x_r)\in \TT^r$ has pairwise distinct elements, $\Vx \in \C^{(K+1)\times r}$ is the Vandermonde matrix whose columns are $\bm{e}_K(x_k)$ for $k\in\{1,\ldots,r\}$, and $\Da \in \R^{r\times r}$ is a diagonal matrix  with $\bm a = (a_1,\ldots,a_r)\in(\RR_{>0})^r$ on the diagonal.
\end{enumerate}
Moreover, when $\bm{\mathrm{T}}$ is rank-deficient (i.e. $r<K+1$), then the decomposition in Item~\ref{item:decomp} is unique up to any permutation applied to the coefficients of $\bm x$ and $\bm a$. When $\bm{\mathrm{T}}$ has full rank, then there are uncountably many $\{x_1,\ldots,x_{K+1}\}\subset\TT$ such that the decomposition holds. Note that $x_{K+1}$ can be arbitrarily chosen in $\TT$.
\end{proposition}

We remark that the decomposition given in~\cite[Theorem 1]{yang2018frequency} is as a sum of $r$-rank one matrices $a_k\bm{e}_K(x_k)\bm{e}_K(x_k)^*$. This can readily be shown to be equivalent to the decomposition in Proposition~\ref{prop:CFP}.

Secondly, the result of~\cite{yang2018frequency}, does not adjudicate on the uniqueness of the decomposition when the Toeplitz matrix $\bm{\mathrm{T}}$ is of full rank. However, by carefully studying the proof of~\cite[Theorem 1]{yang2018frequency}, one can see that the full-rank case is studied in detail (see in particular Equations~(10) to~(12)). We summarized the conclusion at the end of Proposition~\ref{prop:CFP}.

{\footnotesize
\bibliographystyle{IEEEtran}
\bibliography{references}}
\end{document}